\documentclass[12pt]{amsart}

\usepackage{amsmath,amsthm,amssymb}
\usepackage{verbatim}
\usepackage{tikz}
\usepackage{float}
\usepackage{mathtools}
\usepackage{subcaption}
\usepackage{algorithm}
\usepackage{algorithmic}
\usepackage{MnSymbol}
\usepackage{graphicx}
\usepackage[left=2.7cm,top=2.7cm,right=2.7cm,bottom=2.7cm]{geometry}

\newtheorem{theorem}{Theorem}
\newtheorem{lemma}[theorem]{Lemma}
\newtheorem{corollary}[theorem]{Corollary}

\theoremstyle{definition}

\newcommand{\qedclaim}{\hfill $\diamond$ \medskip}

\newcommand{\loc}[1]{\zeta(#1)}

\newcommand{\floor}[1]{\left \lfloor #1 \right \rfloor}

\begin{document}

\title{The localization number of designs}\thanks{The first and second authors were supported by NSERC}

\author[A.\ Bonato]{Anthony Bonato}
\author[M. A.\ Huggan]{Melissa A. Huggan}
\author[T.\ Marbach]{Trent Marbach}
\address[A1, A2, A3]{Ryerson University, Toronto, Canada}
\email[A1]{(A1) abonato@ryerson.ca}
\email[A2]{(A2) melissa.huggan@ryerson.ca}
\email[A3]{(A3) trent.marbach@ryerson.ca}

\begin{abstract}
We study the localization number of incidence graphs of designs. In the localization game played on a graph, the cops attempt to determine the location of an invisible robber via distance probes. The localization number of a graph $G$, written $\zeta(G)$, is the minimum number of cops needed to ensure the robber's capture. We present bounds on the localization number of incidence graphs of balanced incomplete block designs. Exact values of the localization number are given for the incidence graphs of projective and affine planes. Bounds are given for Steiner systems and for transversal designs.
\end{abstract}


\maketitle

\section{Introduction}

In \emph{graph searching}, we consider simplified, combinatorial models for the detection or neutralization of an adversary's activity on a network. Such models often focus on vertex-pursuit games, where agents or cops are attempting to capture an adversary or robber loose on the vertices of a network. The players move at alternating ticks of the clock, and have restrictions on their movements or relative speed depending on the type of game played. The most studied such game is Cops and Robbers, where the cops and robber can only move to vertices with which they share an edge. The cop number is the minimum number of cops needed to guarantee the robber's capture. How the players move and the rules of capture depend on which variant is studied, and these variants are motivated by problems in practice or inspired by foundational issues in computer science, discrete mathematics, and artificial intelligence, such as robotics and network security. For a survey of graph searching see \cite{bp,by,fomin}, and see \cite{BN} for more background on Cops and Robbers.

The cop number of graphs derived from designs was studied in~\cite{BB}. Incidence graphs of projective planes are \emph{Meyniel extremal} in the sense that they have the largest conjectured cop number among connected graphs as a function of their order. Such a connection was made implicitly in the early work of \cite{frankl} (where Meyniel's conjecture first appeared) and made explicit in \cite{p}. Several Meyniel extremal families arising from incidence graphs were discovered in \cite{BB}, including those from oval designs, Denniston designs, and transversal and truncated transversal designs.

The \emph{localization game} is a variant of Cops and Robbers that has received much recent attention, and we study the game on incidence graphs of designs.  We note that our work is the first of its kind in this direction. Before we define the localization game in the next paragraph, we define a \emph{design} $\mathcal{D}=(X,\mathcal{B})$ to be a set of points $X$ and a set of blocks $\mathcal{B}$, where each block is a subset of $X$. A design is \emph{$t$-balanced} when we insist that each $t$-tuple of $X$ occurs in exactly $\lambda$ blocks of $\mathcal{B}$. The parameter $\lambda$ is called the \emph{index} of the $t$-balanced design. The {\em incidence} (or \emph{Levi}) \emph{graph} of a design $\mathcal{D}$ is the bipartite graph with vertex set $X \cup \mathcal{B}$, such that $x \in X$ is adjacent to $B \in \mathcal{B}$ if and only if $x$ lies in $B$.

In the localization game, as in Cops and Robbers, there are two players moving on a connected graph, with one player controlling a set of $k$ \emph{cops}, where $k$ is a positive integer, and the second controlling a single \emph{robber}. Unlike in Cops and Robbers, the cops play with imperfect information: the robber is invisible to the cops during gameplay. The game is played over a sequence of discrete time-steps; a \emph{round} of the game is a move by the cops together with the subsequent move by the robber. The robber occupies a vertex of the graph, and when the robber is ready to move during a round, he may move to a neighboring vertex or remain on his current vertex. A move for the cops is a placement of cops on a set of vertices (note that the cops are not limited to moving to neighboring vertices). At the beginning of the game, the robber chooses his starting vertex. After this, the cops move first, followed by the robber; thereafter, the players move on alternate time-steps. Observe that any subset of cops may move in a given round. In each round, the cops occupy a set of vertices $u_1, u_2, \ldots , u_k$ and each cop sends out a \emph{cop probe}, which gives their distance $d_i$, where $1\le i \le k$, from $u_i$ to the robber. Hence, in each round, the cops determine a \emph{distance vector} $(d_1, d_2, \ldots , d_k)$ of cop probes, which is unique up to the ordering of the cops. Note that relative to the cops' position, there may be more than one vertex $x$ with the same distance vector. We refer to such a vertex $x$ as a \emph{candidate}. For example, in an $n$-vertex clique with a single cop, so long as the cop is not on the robber's vertex, there are $n-1$ many candidates. The cops win if they have a strategy to determine, after finitely many rounds, a unique candidate, at which time we say that the cops {\em capture} the robber. If there is no unique candidate in a given round, then the robber may move again. The cops may move to other vertices in the next round resulting in an updated distance vector. The robber wins if he is never captured.

For a connected graph $G$, define the \emph{localization number} of $G$, written $\loc{G}$, to be the least integer $k$ for which $k$ cops have a winning strategy over any possible strategy of the robber (that is, we consider the worst case that the robber a priori knows the entire strategy of the cops). As placing a cop on each vertex gives a distance vector with unique value of $0$ on the location of the robber, the localization number is well-defined.

The study of the localization game is motivated by a real-world tracking problem with mobile receivers and a cell phone user. The receivers are placed in various locations, and the user is in motion and is only detectable by the strength of their signal to the receivers (measured by their distance to the receivers). The receivers, who do not know the user's location, may appear anywhere and relocate over time. The goal is to uniquely determine the location of the user. See, for example, \cite{bahl}. The localization game was first introduced for one cop by Seager~\cite{seager1,seager2} and was further studied in \cite{BK,BDELM,car}. Interestingly, the localization number is related to the metric dimension of a graph, in a way that is analogous to how the cop number is bounded above by the domination number. The \emph{metric dimension} of a graph $G$, written $\mathrm{dim}(G)$, is the minimum number of cops needed in the localization game so that the cops can win in one round; see \cite{hm,slater}. Hence, $\loc{G} \le \mathrm{dim}(G)$, but in many cases this inequality is far from tight. The bound of $\loc{G} \le \floor{\frac{(\Delta +1)^2}{4}} +1$, where $\Delta$ is the maximum degree of $G$, was shown in \cite{has}. In \cite{nisse1}, Bosek et al.\ showed that $\loc{G}$ is bounded above by the pathwidth of $G$ and that the localization number is unbounded even on graphs obtained by adding a universal vertex to a tree. They also proved that computing $\loc{G}$ is \textbf{NP}-hard for graphs with diameter 2, and they studied the localization game for geometric graphs. The centroidal localization game was considered in \cite{nisse2}, where it was proved, among other things, that the centroidal localization number (and hence, the localization number) of outerplanar graphs is at most 3.  In \cite{DFP}, the localization number was studied for binomial random graphs with diameter 2, with further work done in this direction was done in \cite{DEFMP,DFP}. Bonato and Kinnersley~\cite{BK} studied the localization number of graphs based on their degeneracy. In ~\cite{BK}, they resolved a conjecture of \cite{nisse1} relating $\loc{G}$ and the chromatic number; further, they proved that the localization number of outerplanar graphs is at most 2, and they proved an asymptotically tight upper bound on the localization number of the hypercube.

We note first that in keeping with the style used to discuss the metric dimension of incidence graphs, we abuse language and \emph{refer to the localization number of the incidence graph of a design as the localization number of the design.} For example, we reference below to the localization number of a projective plane.

The paper is organized as follows. In Section~\ref{bibd}, we provide upper and lower bounds on the localization number of balanced incomplete block designs for general $\lambda$ and in the special case of $\lambda = 1$. We present proof techniques that are used repeatedly throughout this paper. In Section~\ref{ss: projective and affine planes}, we determine the exact value of the localization number of projective and affine planes. In Section~\ref{ss: steiner systems}, we give bounds on the localization number of Steiner systems. We finish with a discussion of the localization number of transversal designs and state several open problems.

We summarize our results for the localization numbers of designs in the chart below. All the graphs in the chart are incidence graphs $G$. The columns list the design, bounds or exact values of $\zeta(G)$, and a reference to the appropriate theorem or corollary in the paper. Note that the function $f(G)$ in the first line of the chart is defined before Theorem~\ref{thm:BIBD_gen_upper}, and the value $d$ in the third line is defined in the statement of Theorem~\ref{thm:genLowerBIBD}.
\vspace{0.1in}
\begin{center}
\begin{tabular}{|l||c|c|}
\hline
\textbf{Design} & \textbf{Bounds or values} & \textbf{Reference}\\
\hline\hline
\rm{BIBD}$(v,b,r,k,\lambda)$, $2 \leq \lambda \leq r-1$ & $\zeta(G) \leq f(G) + r+1$ & Theorem~\ref{thm:BIBD_gen_upper} \\ \hline
\rm{BIBD}$(v,b,r,k,1)$ & $\zeta(G) \le 2r+k-3$ & Corollary~\ref{corrr} \\ \hline
\rm{BIBD}$(v,b,r,k,1)$, $k<r$ & $d < \zeta(G)$ & Theorem~\ref{thm:genLowerBIBD} \\ \hline
Symmetric \rm{BIBD}$(v,b,r,k,1)$ & $\zeta(G) = k$ & Theorem~\ref{thm:pp} \\ \hline
Projective plane of order $q$ & $\zeta(G) =q+1$ & Theorem~\ref{thm:pp} \\ \hline
\rm{BIBD}$(k^2, k^2 + k,k+1,k,1)$, $k \geq 3$ & $\zeta(G) =k$ or $k+1$ & Corollary~\ref{cor:near-symmetricBIBD} \\ \hline
Affine plane of order $q$ & $\zeta(G) =q$ & Theorem~\ref{affinet} \\ \hline
\rm{STS}$(v)$, $v>9$ & $\lfloor \frac{v-2}{8} \rfloor \le \zeta(G) \le \frac{v+1}{2}$ & Corollary~\ref{stsc}, Theorem \ref{stsv} \\ \hline
\rm{STS}$(v)$ & $\zeta(G) \le (1+o(1))v/3$ &  Theorem~\ref{thm:upperSTSAlon} \\ \hline
$S(3,4,v)$, $v \geq 6$ & $\zeta(G) \le v-3$ & Theorem~\ref{stsq} \\ \hline
$S(t,k,v)$ & $\zeta(G) \le (1+o(1))v/k$ & Theorem~\ref{stsss} \\ \hline
$TD(k,n)$ & $\zeta(G) \leq n+k-4$ & Theorem~\ref{td} \\ \hline
\end{tabular}
\end{center}
\vspace{0.1in}

Throughout, all graphs considered are simple, undirected, connected, and finite. For a general reference for graph theory see~\cite{West}, and for a reference about designs see~\cite{Stinson2004}.

\section{Localization number of BIBDs}\label{bibd}
Our focus in this section is to provide bounds on the localization numbers of BIBDs, in both the case $\lambda =1$ and otherwise. Throughout the paper, there are three recurring proof techniques, and we begin by discussing these.

When playing the localization game, a \emph{resolving set} of a graph is a set $S$ of vertices such that if the cops play on all the vertices of $S$ in their first turn, there is only one candidate vertex. A \emph{delayed-resolving set} of a graph is a set $S$ of vertices such that if the cops play on all the vertices of $S$, the set of candidates is an independent set. Hence, if the cops play on the same delayed-resolving set during each turn, then the robber's only technique to avoid capture is to remain stationary on the independent set.

The first technique we describe is \emph{scanning} a graph, as summarized in the following lemma.
\begin{lemma}[Scanning Lemma]\label{lem:graphscanning}
If a graph $G$ has a delayed-resolving set $S$, then $\zeta(G) \leq |S|+1$.
\end{lemma}
\begin{proof}
Place $|S|$ cops on the vertices of $S$ for all cop turns, and during each cop turn, place the remaining cop on a previously unvisited vertex of $G$. After a finite number of moves, this cop will visit the vertex occupied by the robber. To avoid capture by this cop, the robber must leave his vertex before this cop is placed on his vertex. When the robber moves, he will be located by the cops.
\end{proof}

In the game of Cops and Robbers, the notion of robber territory is a useful tool, especially when studying the cop number of planar graphs and graphs on surfaces; see \cite{bm}. Roughly put, the robber territory is an induced subgraph where the robber is safe from capture. For the localization game, we define the \emph{robber territory} as follows. The robber territory is initialized to be $T_{0} = V(G)$. After the cops have moved on the cops $i$th turn, we define $T'_i$ to contain those vertices that are in $T_{i-1}$ or the neighbors of a vertex in $T_{i-1}$.  The vertices in $T'_i$ can be partitioned into classes such that each class contains exactly those vertices of $T'_i$ with identical distance vectors.  That is, if the robber moves to a class $B \subseteq T'_i$, then the cops can identify that the robber is on $B$, but cannot distinguish which of the vertices of $B$ he is on unless $B$ is a singleton. The class that the robber currently resides on is defined as $T_i$.  As the robber has perfect information, he is able to choose which of the classes of $T'_i$ is used for $T_i$.

The second technique we employ is \emph{decreasing the robber territory}. We can show that over a certain number of turns, the cops are able to ensure that the robber territory strictly decreases in size. Thus, as turns progress, the number of vertices in the robber territory will reduce and so the robber will eventually be captured.

Finally, we consider \emph{maintaining the robber territory}, where we show that if the robber territory contains a small number of vertices, then the robber can move in such a way that the robber territory will always be sufficiently large (at least size two) to avoid capture during every round. Hence, the robber wins the game.

We begin with a general result about the localization number of $2$-balanced designs with index $\lambda=1$.

\begin{theorem} \label{thm:2designs}
If a $2$-balanced design has index $\lambda=1$, then its incidence graph $G$ with vertex partition $X\cup \mathcal{B}$ satisfies $$\zeta(G) \leq |N(u)| + |N(u')| +|N(v)| -3,$$ for each $u,u' \in X$, $u\neq u'$ and $v \in \mathcal{B}$ such that $N(u)\cap N(u') = \{v\}$ and such that $|N(u)|\geq 2$ and $|N(u')|\geq 2$.
\end{theorem}
\begin{proof}
We use the scanning technique described above. We will show that $S=N(u) \cup N(u') \cup N(v) \setminus \{u,u',v\}$ is a delayed-resolving set of size $|N(u)| + |N(u')| +|N(v)| -4$, and the result then follows by the Scanning Lemma.
In particular, we will show that if the robber is ever on $X$, then he is identified by the cops on $S$. The remaining vertices that are not resolved by $S$ are a subset of $\mathcal{B}$, and so form an independent set as required.

If the robber is on $X \setminus N(v)$, then there is a unique path of length $2$ from $u$ to the robber $R$, which transverses some vertex $v_1$ in $N(u) \setminus\{v\}$. There is also a unique path of length $2$ from $u'$ to $R$, which transverses some vertex $v_2$ in $N(u') \setminus\{v\}$. We then have that there are exactly two cops that will probe a distance of $1$ from the robber, which are the cops on $v_1$ and $v_2$. Further, there is only one vertex in their common neighborhood (as $\lambda=1$, two blocks cannot intersect in more that one point, so the common neighborhood of two block vertices contains zero or one point), so the robber's location is discovered.

If the robber is on $N(v) \setminus\{u,u'\}$, then exactly one cop on $N(v) \setminus\{u,u'\}$ probes a distance of $0$, and the robber's location is found. Suppose that the robber is on one of $u$ or $u'$; without loss of generality, say the robber is on $u.$ All cops on $N(u)\setminus\{v\}$ probe a distance of $1$ and all cops on $N(u') \setminus \{v\}$ probe a distance of $2$, which only occurs for $u$. \end{proof}

An important family of $2$-balanced designs are \emph{balanced incomplete block designs}, written $\text{BIBD}(v,b,r,k,\lambda)$, which have a point set $X$ of size $v$ and a block set $\mathcal{B}$ of size $b$ such that the size of each block is $k$, each point occurs in $r$ blocks, and each pair of points occurs in $\lambda$ blocks. Note that $v \leq b$, $k \geq r$, $vr=bk$, and $\lambda(v-1) = r(k-1)$.  A $\text{BIBD}$ is \emph{symmetric} if $v=b$, in which case $r=k$ and $v = \lambda k(k-1) + 1$. A $\text{BIBD}$ is \emph{simple} if it has no repeated blocks. Although we will focus on the case when $\lambda=1$, there are upper and lower bounds that we find in the general case.

Let $G$ be the incidence graph of a $\text{BIBD}(v,b,r,k,\lambda)$ with vertex partitions $X$ and $\mathcal{B}$ corresponding to the points and blocks, respectively. Note that $X$ has $v$ vertices, $\mathcal{B}$ has $b$ vertices, each vertex in $X$ has degree $r$, each vertex in $\mathcal{B}$ has degree $k$, and for each $u_1,u_2 \in X$ we have $|N(u_1) \cap N(u_2)|=\lambda$. The distance between any two vertices of $X$ is always $2$, and the distance between vertices of $\mathcal{B}$ is either $2$ or $4$.

We next introduce a function that will be useful in bounding the localization number of BIBDs.
Fix some vertex $u \in X$.
Decompose the vertices of $X\setminus \{u\}$ into sets $X_1, \ldots, X_{\alpha}$ such that $\{u\}\cup X_i$ occurs in $\lambda$ blocks of $\mathcal{B}$.
If there is some $x \in X_j$, such that $\{u\}\cup X_i \cup \{x\}$ occurs in $\lambda$ blocks of $\mathcal{B}$, then $X_i$ and $X_j$ can be combined.
It follows that there is a unique partition with parts of maximum size, and we assume that $X_1, \ldots, X_{\alpha}$ is this maximum partition.
Each part in this partition will require $|X_i|-1$ extra cops in order to determine the location of the robber on $X_i$.
In total, we will require  $f(u) = \sum_{i}  (|X_i| - 1)$ additional cops if we use $u$ to create this partition.
We take $f(G) = \min_{u\in X} f(u)$, which will be the minimum number of additional cops that are required using the particular strategy we employ.

Two non-isomorphic \rm{BIBD}$(7,21,9,3,3)$ are given in Figure~\ref{fig:bibdLambda3}; see \cite{HCD}. The incidence graph of the first has $f(G) = 3$ while the second has $f(G)=1$.
Note that in general, if no triple occurs $\lambda$ times in the design, then the corresponding incidence graph $G$ has $f(G)=0$.

\begin{figure}
\begin{center}
\begin{tabular}{ccc}
\hline
{000000000111111222222}\\
{111333555333444333444}\\
{222444666555666666555}\\
\hline
\end{tabular}
\quad
\begin{tabular}{ccc}
\hline
{000000000111111222222}\\
{111333555333444333444}\\
{222444666556566566556}\\
\hline
\end{tabular}
\end{center}
\caption{Two non-isomorphic $\mathrm{BIBD}(7,21,9,3,3)$, where blocks are represented as columns.}  \label{fig:bibdLambda3}
\end{figure}

Using the function $f$, we have the following upper bound for the localization number of BIBDs.
\begin{theorem} \label{thm:BIBD_gen_upper}
If $G$ is an incidence graph of a \rm{BIBD}$(v,b,r,k,\lambda)$ with $2 \leq \lambda \leq r-1$, and function $f$ defined above, then we have that
\[ \zeta(G) \leq f(G) + r+1.
\]
\end{theorem}
\begin{proof}
We use the scanning technique described above.
We construct a delayed-resolving set $S$.
Let $u\in X$ be a vertex such that $f(u)=f(G)$ and $X=\{X_i\}$ is the unique partition with parts of maximum size used to define $f(u)$.
For each $i$, we place $|X_i|-1$ vertices from $X_i$ in $S$.
We also add the $r$ vertices of $N(u)$ to $S$.
Therefore, $S$ has size $f(u) +r$, and so once we show that $S$ is a delayed-resolving set, the result then follows by the Scanning Lemma.
In particular, we will show that if the robber is ever on $X$, then he is identified by the cops on $S$.
The remaining vertices that are not resolved by $S$ are a subset $\mathcal{B}$, and so form an independent set as required.

Consider the results if we play $f(u) +r$ cops on the vertices of $S$.
If the robber is on point $u$, then all $r$ cops on $N(u)$ will probe a distance of $1$ to the robber, and so he is captured (note that $r > \lambda$).
If the robber $R$ is on $x \in X_i$, then by the properties of the design, there are $\lambda$ paths of length $2$ from $x$ to $u$, where the intermediary points of the paths are exactly the set of $\lambda$ blocks that contain $\{u\} \cup X_i \subseteq N(u)$.
Thus, the $\lambda$ cops in blocks containing $\{u\} \cup X_i$ indicate a distance of $1$ to the robber and the cops in the other blocks of $N(u)$ indicate a distance of $3$ to the robber as all other blocks in $N(u)$ do not contain any element of $X_i$ by definition of $\lambda$.
This implies that the cops know that the robber resides on one of the vertices of $X_i$. As we have also placed $|X_i|-1$ cops on the vertices of $X_i$, either one cop will probe a distance of $0$ and the robber is captured, or all the cops on $X_i$ will probe a distance of $2$ and the robber resides on the unique vertex of $X_i$ without a cop.
Hence, if the robber is on a vertex of $X,$ then he is captured immediately.
Thus, $S$ is a delayed-resolving set, which completes the proof.
\end{proof}

The \emph{degeneracy} of a graph $G$ is the maximum, over all subgraphs $H$ of $G,$ of the minimum degree of $H.$ There is an immediate lower bound for $\text{BIBD}$s with general $\lambda$, based on degeneracy results from \cite{BK}.

\begin{theorem} \label{thm:BIBD_gen_lower}
If $G$ is the incidence graph of a \rm{BIBD}$(v,b,r,k,\lambda)$, then we have that $$\zeta(G) \geq \log_2(k).$$
\end{theorem}

\begin{proof}
The incidence graph of a \rm{BIBD}$(v,b,r,k,\lambda)$ is bipartite and has degeneracy $k$. The result follows by Theorem~2.3 of \cite{BK}. \end{proof}

Observe that if $G=(X\cup \mathcal{B},E)$ is the incidence graph of a $\text{BIBD}(v,b,r,k,1)$, then the following properties hold.
\begin{enumerate}
\item The graph $G$ is $(r,k)$-\emph{biregular}: all vertices have degree $r$ or $k$.
\item For all $u_1,u_2 \in X$, we have that $|N(u_1) \cap N(u_2)|=1,$ and for all $v_1,v_2 \in \mathcal{B},$ we have that $|N(v_1) \cap N(v_2)|\leq 1$.
\item The graph $G$ has girth 6; that is, its smallest cycle is of order 6.
\item If the BIBD is symmetric, then for all $v_1,v_2 \in \mathcal{B}$, we have that $|N(v_1) \cap N(v_2)|=1$.
\end{enumerate}

We may also restrict our attention to $\rm{BIBD}$s with $\lambda = 1$.
In this case, Theorem~\ref{thm:2designs} yields the following upper bound.

\begin{corollary}\label{corrr}
If $G$ is the incidence graph of a \rm{BIBD}$(v,b,r,k,1),$ then we have that
\[ \zeta(G) \leq 2r+k-3.
\]
\end{corollary}

We also have an improved lower bound on the localization number of $\rm{BIBD}$s with $\lambda=1$.

\begin{theorem} \label{thm:genLowerBIBD}
Let $G$ be the incidence graph of a \rm{BIBD}$(v,b,r,k,1)$ with $k<r$ and let $\alpha = \min\{k,r-k\}$ and let $d$ be some positive integer with
$d \leq \frac{r-2}{\alpha}$ such that $d < \max \left\{k, \frac{2r(k-1)-2}{k+1+2(k-1)\alpha}\right\}$.
We then have that $\zeta(G) > d.$
\end{theorem}
\begin{proof}
Suppose for contradiction that $d$ cops were sufficient to capture the robber. The technique we use is to maintain the robber territory.
We will show that if there are at least two vertices of $X$ in the robber territory during some turn, then the robber can move such that the robber territory on the next round contains at least $r -d\alpha \geq 2$ vertices of $\mathcal{B}$. In the subsequent round, the robber can move such that the robber territory on the next round contains at least $2$ vertices of $X$.
This will complete the proof by induction, noting that the initial case follows immediately, as the robber can play his first move by assuming he was already on a vertex of $X$ and moving to a vertex of $\mathcal{B}$.

\emph{From $X$ to $\mathcal{B}$:}
Suppose that the robber territory $T_i$ contains only $u\in X$ (but the robber was not caught in the previous round). We analyze the distance from the placement of any cop to the vertices in $T'_{i+1} = N(u)$. If a cop is in $X \setminus \{u\}$, then she has distance $1$ to exactly one vertex in $N(u)$ and distance $3$ to all other vertices in $N(u)$. If a cop is on $u$, then she has distance $1$ to all vertices in $N(u)$. If a cop is in $\mathcal{B} \setminus N(u)$, then she has distance $2$ to $k$ vertices in $N(u)$ and distance $4$ to all other vertices in $N(u)$. If a cop is in $N(u)$, then she has distance $0$ to the unique vertex in $N(u)$ that she is on, and distance $2$ to all other vertices in $N(u)$.

For the vertices in $(X \setminus \{u\}) \cup N(u)$, a cop at this location has distance $d\in \{0,1\}$ to one vertex of $N(u)$ and distance $d+2$ to the other $r-1$ vertices of $N(u)$.
Hence, a cop at this location can distinguish only one vertex of $N(u)$ as being separated from the others. On $u$, a cop has distance $1$ to all vertices in $N(u)$ and so cannot distinguish the vertices of $N(u)$ at all.
However by playing on  $\mathcal{B} \setminus N(u)$, a cop can distinguish $k$ vertices of $N(u)$ (which have distance $2$ to the robber) from $r-k$ vertices of $N(u)$ (which have distance $4$ to the robber).
If we played all $d$ cops on vertices of $\mathcal{B} \setminus N(u)$, then there would be  $r -d\alpha$ vertices of $N(u)$ that have the same distance to all cops (recall that $\alpha = \min\{k,r-k\}$).
If the cops played on other vertices, there would be even more vertices of $N(u)$ that have the same distance to all cops.
As $r -d\alpha \geq 2$ by assumption, the robber can move such that there are at least two vertices in the new robber territory $T_{i+1}$.

\emph{From $\mathcal{B}$ to $X$:}
Suppose that the robber territory $T_{i+1}$ contains at least $r -d\alpha $ vertices in $N(u)$ with $u \in X$.
We analyze the distance from the placement of any cop to the vertices in $T'_{i+2} = N(T_{i+1}) \setminus \{u\}$, assuming that the robber will move to one of these vertices.
We assume that the robber does not move to $u$, as this move makes it straightforward to identify his location, and so disallowing this robber move can only strengthen the cops' strategy. We can make this assumption as it only weakens the robber's  strategy, and we will show that the robber can win regardless.
We examine each case for the cops.

If a cop is in $T_{i+1}$, then she has distance $3$ to all vertices in $T'_{i+2}$.
If a cop is in $N(u) \setminus T_{i+1}$, then she has distance $1$ to at most $k-1$ vertices in $T'_{i+2}$ and in particular at most one in $N(v)\setminus \{u\}$ for each $v \in T_{i+1}$, and distance $3$ to all other vertices in $T'_{i+2}$. If a cop is in $\mathcal{B} \setminus N(u)$, then she has distance $1$ to at most $k$ vertices in $T'_{i+2}$, but at most one in $N(v)\setminus \{u\}$ for each $v \in T_{i+1}$, and distance $3$ to all other vertices in $T'_{i+2}$. If a cop is on $u$, then she has distance $2$ to all vertices in $T'_{i+2}$. If a cop is in $T'_{i+2}$, then she has distance $0$ to the unique vertex in $T'_{i+2}$ that she is on, and distance $2$ to all other vertices in $T'_{i+2}$. If a cop is in $X \setminus T'_{i+2}$, then she has distance $2$ to all vertices in $T'_{i+2}$.
From here, we provide two observations that show that the robber can escape when $d$ is below some given function on the parameters of the $\rm{BIBD}$, where each observation provides a different such function.

As a first observation, we now show that if $d < k$, then the robber can escape.
Suppose that the cops knew that the robber resides on $N(v) \setminus \{u\}$ for some $v \in T_{i+1}$.
Each cop can only uniquely distinguish at most one vertex of the $k$ vertices in $N(v)$, and so at least $k-1$ cops are needed to distinguish the exact location of the robber, and these cops must reside on $T'_{i+2} \cup (\mathcal{B} \setminus N(u))$.
If instead the cops knew that the robber resides on $(N(v) \cup N(v') )\setminus \{u\}$ for some $v,v' \in T_{i+1}$ and $v \neq v'$, then it still holds by the same argument that at least $k-1$ cops must reside on $T'_{i+2} \cup (\mathcal{B} \setminus N(u))$. However, then there are two vertices (one in $N(v)\setminus \{u\}$  and one in $N(v')\setminus \{u\}$) that are indistinguishable by the cops.
Hence, $k-1$ cops are insufficient to capture the robber.

As a second observation, we now show that if $d < \frac{2r(k-1)-2}{k+1+2(k-1)\alpha}$, then the robber can escape.
Note that $|T'_{i+2}| = (k-1)(r-d\alpha)$.
If there are two vertices in $T'_{i+2}$ that have distance $2$ or $3$ to all cops, then the distance vectors for these two vertices would be the same, and the cops cannot distinguish them.
Thus, as we are assuming the robber is captured with $d$ cops, at most one vertex in  $T'_{i+2}$ that has distance $2$ or $3$ to all cops.
If there are $d+1$ vertices in $T'_{i+2}$ that have distance $2$ or $3$ to all cops except one, then by the pigeonhole principle there are two vertices in  $T'_{i+2}$ that have distance less than $2$ to the same cop and distance $2$ or $3$ to all other cops. These distance vectors for these two vertices would be the same, and the cops cannot distinguish them.
Thus, assuming the robber can be captured, at most $d$ vertices in  $T'_{i+2}$  have distance $2$ or $3$ to all cops except one. The cops cannot hope to be more efficient than this, so we assume that each cop is distance less than two to its own unique vertex of $T'_{i+2}$.
We then have that there are a set of $(k-1)(r-d\alpha) -1-d$ vertices of $T'_{i+2}$ that must have distance less than $2$ to at least two cops.
However, each cop can have distance less than $2$ to at most $k-1$ vertices in this set. Let $\mathcal{B}'$ be the blocks of $\mathcal{B}$ that contain a set of $d$ cops, and let $X'$ be the points in $X$ that have two neighbors in $\mathcal{B}'$. We then have $|X'| = (k-1)(r-d\alpha) -1-d$ and $|\mathcal{B}'|=d.$ The vertices of $\mathcal{B}'$ have degree $k-1,$ and the vertices of $X'$ have degree at least $2.$ Thus, we have that $2[(k-1)(r-d\alpha) -1-d] \leq d(k-1)$. It then follows that $d\geq \frac{2r(k-1)-2}{k+1+2(k-1)\alpha}$ is a requirement for the cops to be able to capture the robber in this case, contradicting our assumptions about $d$.  Hence, if $d < \max \left\{ k, \frac{2r(k-1)-2}{k+1+2(k-1)\alpha}\right\}$, then the robber can move from $\mathcal{B}$ to $X$ and avoid capture during the next cop move.
\end{proof}

\section{Projective and affine planes}\label{ss: projective and affine planes}

We now consider well-known finite geometries including projective and affine planes. Incidence graphs of projective planes of order $q\ge 23$ were shown to have metric dimension $4q-4$ in \cite{ht}, and so this provides upper bounds for their localization numbers. For recent work on the metric dimension of incidence graphs of symmetric designs, see \cite{bailey}. In Theorem~\ref{thm:pp}, we prove that projective planes of order $q$ have localization number $q+1.$

We first establish a lower bound on the localization number of symmetric BIBDs. This proof runs similar to Theorem \ref{thm:genLowerBIBD}, with two differences. First, by considering two vertices in the robber territory when moving from $X$ to $\mathcal{B}$, we can increase the lower bound by one. Second, due to the symmetric property, the direction going from $X$ to $\mathcal{B}$ and going from $\mathcal{B}$ to $X$ are identical.

\begin{theorem} \label{thm:ppLowerBIBD}
If $G$ is an incidence graph of a symmetric \rm{BIBD}$(v,b,r,k,1)$ (so $v=b$ and $r=k$), then $\zeta(G) > k-1.$
\end{theorem}
\begin{proof}
To derive a contradiction, we assume using $k-1$ cops is sufficient to capture the robber, and we play with $k-1$ cops.
Suppose at some round during the game play that the robber territory $T_i$ has at least two vertices $u_1,u_2$ in $X$. For the initial case, note that $T_0$ contains all vertices, and so contains at least two vertices in $X$.
The robber may stay on $u_1,u_2$, and so some cop must be used to distinguish $u_1$ and $u_2$. However this is done, the cop that is used to distinguish these two vertices will have constant distance to all vertices in either $N(u_1)$ or $N(u_2)$.
Without loss of generality, we assume this cop has constant distance to the vertices in $N(u_1)$.
Each position that the cops can take can distinguish at most one vertex in $N(u_1)$, and so the remaining $k-2$ cops can distinguish at most $k-2$ vertices in $N(u_1)$, and so there are two remaining vertices in $N(u_1)$ that are not able to be distinguished, and so the robber territory $T_{i+1}$ contains at least two vertices of $\mathcal{B}$. By induction, the proof is complete.
\end{proof}

The lower bound just provided is in fact tight, as will be shown in the next theorem.
The proof used in the next theorem uses a two-step induction.
Let $u$ be a fixed vertex that could be in $X$ or $\mathcal{B}$.
To explain this proof technique, let $P(\alpha)$ be the condition that cops know that the robber is on a set of $k-1-\alpha$ vertices of $N(v) \setminus \{u\}$ for some $u \in X$, $v \in N(u)$, and $k-1 - \alpha \geq 3$, and similarly let $Q(\alpha)$ be the condition that the cops identify that the robber is on a set of $r-1-\alpha$ vertices of $N(v) \setminus \{u\}$ for some $u \in \mathcal{B}$, $v \in N(u)$, and $r-1 - \alpha \geq 3$.
We initialize by showing that $P(0)$ is true.
We then show that if $P(\alpha)$ is true, then $Q(\alpha+2)$ is true, which we call Step 1.
After this, we show that if $Q(\alpha)$ is true, then $P(\alpha+2)$ is true, which we call Step 2.
We then proceed in an inductive-like fashion, applying Step 1 followed by Step 2, until we have that $P(k-2)$ is true.
If $P(k-2)$ is true, then there is only one vertex that the robber could reside on, and so the cops have identified the location of the robber.
A modification of this two-step induction technique will also be used in the proofs of Theorems~\ref{thm:nearSymUpperBound} and \ref{affinet}.

\begin{theorem}\label{thm:pp}
If $G$ is the incidence graph of a symmetric \rm{BIBD}$(v,b,r,k,1)$ and $k \geq 3$, then we have that $\zeta(G) =  k$.
\end{theorem}
\begin{proof}
By Theorem~\ref{thm:ppLowerBIBD}, it remains to be shown that $\zeta(G)\leq  k$.
We give a cop strategy that shows that $k$ cops are sufficient to capture the robber.
We use the technique of decreasing the robber territory.

Let $u \in X$.
For their first turn, the $r=k$ cops can be positioned such that the robber is known by the cops to be on a set of $k-1$ vertices $N(v)\setminus \{u\} \subseteq X$ for some $v \in N(u)$ or on a single vertex.
(Note that although $r=k$, it will be useful in later proofs to use $r$ when referencing the neighborhood of a vertex in $X$, and $k$ when referencing the neighborhood of a vertex in $\mathcal{B}$.)
To do this, we place $r-1$ cops on $r-1$ of the $r$ vertices of $N(u)$ (that is, leaving exactly one vertex without a cop), and place the last cop on $\mathcal{B} \setminus N(u)$.
While the robber remains on $\mathcal{B}$, the cops implement the scanning technique to either (i) force the robber onto $X$, or
(ii) after all vertices in $\mathcal{B}$ have been visited and the robber is not forced onto $X$, capture the robber on the vertex without a cop of $N(u)$.

The robber then must at some point move to $X$.
The cops then take their next turn, following the above strategy (they do not know the robber has moved to $X$ until after they have probed).
If all cops (there must be at least two) on $N(u)$ probe a distance of $1$ to the robber, then the cops identify the robber as residing on $u$.
If only one cop $C_i$ on $N(u)$ probes a distance of $1$ to the robber, then the cops identify the robber as residing on $N(v)\setminus \{u\}$, where $v =C_i$.
If all cops on $N(u)$ probe a distance of $3$ to the robber, then the cops identify the robber is residing on $N(v)\setminus \{u\}$ where $v \in N(u)$ is the unique vertex in $N(u)$ without a cop.
Therefore, we have that the robber is identified by the cops to be on a set of $k-1$ vertices $N(v)\setminus \{u\} \subseteq X$ for some $v \in N(u)$ or on a single vertex.
We now proceed by induction.

\emph{Step 1: From $X$ to $\mathcal{B}$.}
Suppose that the cops identify that the robber is on a set of $k-1-\alpha$ vertices $A \subseteq N(v) \setminus \{u\}$ for some $u \in X$, $v \in N(u)$, and $k-1 - \alpha \geq 3$. Here, $\alpha$ is an integer such that $\alpha \geq -1$. On the first application of Step 1, we have $\alpha = 0$. The robber takes his turn.
The cops place $k-2 - \alpha$ cops on the $k-1 -\alpha$ vertices of $A$ without repetition, and $\alpha +2$ cops on vertices of $N(v') \setminus \{u\}$, where $v' \in N(u) \setminus \{v\}$.
If the robber remains on $A$, then he is caught, as if a cop on $A$ does not probe $0$, then they all probe $2$, and the robber is discovered to be on the unique vertex of $A$ without a cop.
If the $k-2-\alpha \geq 2$ cops on $A$ all probe a distance of $1$, then the robber is on $v$.
Otherwise, suppose the robber moved from $u' \in X$ to $N(u')$. If only the cop $C_i$ on $A$ probes a distance of $1$, then the cops discover $u' = C_i$.
Otherwise, all cops on $A$ probe a distance of $3$, and $u'$ is the unique vertex of $A$ without a cop.
If we know the robber is in $N(u')\setminus \{v\}$ and a cop $C_j$ on $N(v') \setminus \{u\}$ probes a distance of $1$, then the robber is uniquely identified.
To have avoided capture (since the robber knows the cops move in advance), the robber must have moved to one of the $r-1-(\alpha+2)$ vertices of $N(u')\setminus \{v\}$ that is not identifiable with the help of the cops in $N(v') \setminus \{u\}$.
The cops, on their turn, identify that the robber is on a set of $r-1 - (\alpha+2)$ vertices $A' \subseteq N(u') \setminus \{v\}$ for some $v \in \mathcal{B}$, $u' \in N(v)$.

\emph{Step 2: From $\mathcal{B}$ to $X$.}
By an analogous argument, suppose that the cops identify that the robber is on a set of $r-1-\alpha$ vertices $A \subseteq N(v) \setminus \{u\}$ for some $u \in \mathcal{B}$, $v \in N(u)$, and $r-1 - \alpha \geq 3$.
The cops, on their turn, identify that the robber is on a set of $k-1 - (\alpha+2)$ vertices $A' \subseteq N(u') \setminus \{v\}$ for some $v \in X$, $u' \in N(v)$.

We repeatedly apply Steps 1 and 2. The process terminates when one of two conditions apply: either $k-1-\alpha < 3$ at the start of Step 1, or when $r-1-\alpha < 3$ at the start of Step 2.
In either case, the cops can add an extra vertex to the set of vertices that the cops cannot distinguish, thereby decreasing $\alpha$ by one. We note that $\alpha$ becomes negative when we add this vertex in the case that $k=3$.
This allows us to continue to apply the steps, and we do so until Step 2 is executed with  $r-1 - \alpha = 3$ (and so $\alpha = k-4$).
In this case, after Step 2 finishes, the cops identify that the robber is on a single vertex as $k-1 - (\alpha+2)=1$.
Thus the cops have located the robber.
\end{proof}

Theorem~\ref{thm:pp} implies that projective planes of order $q$ have localization number $q+1.$ As all BIBDs have $r \geq k$, the symmetric $\rm{BIBD}$s with $r=k$ are the BIBDs with the largest $k$ value with respect to $r$. We next consider $\text{BIBD}$s that almost attain the bound; namely, those with $r =k+1$, in which case $(v,b,r,k, 1) = (k^2, k^2 + k,k+1,k,1)$.

\begin{theorem} \label{thm:nearSymUpperBound}
Let $G$ be the incidence graph of a \rm{BIBD}$(v, b,r,k,1)$ with $r = k+1$ and $k \geq 3$. We then have that $\zeta(G)\leq  k+1$.
\end{theorem}
\begin{proof}
We use the technique of decreasing the robber territory.
The initial set up and Step 1 are analogous to that of Theorem \ref{thm:pp}, and so are omitted. However, we need a slightly different proof for Step 2. Note that $\alpha$ is a parameter that is initialized as $\alpha = 0$, and we have that $\alpha \geq -1$. As before, Step 1 decreases the robber territory from a subset of $X$ of cardinality $k-1-\alpha$ to a subset of $\mathcal{B}$ of cardinality $r-1-(\alpha+2)$.

\emph{Step 2: From $\mathcal{B}$ to $X$.} Suppose that the cops identify that the robber is on a set of $r-1-\alpha$ vertices $A \subseteq N(v) \setminus \{u\}$ for some $u \in \mathcal{B}$, $v \in N(u)$, and $r-1 - \alpha \geq 3$.
The robber takes his turn. If we label the vertices $N(v)$ as $v_1, v_2, \ldots, v_r$, then we can partition $X\setminus\{v\}$ into the disjoint sets $X_i$, where $X_i = N(v_i)\setminus \{v\}$.
The cops place $r-2 - \alpha$ cops on the $r-1 -\alpha$ vertices of $A$ without repetition, and $\alpha +2$ cops on vertices of $N(v') \setminus \{u\}$ without repetition, where $v' \in N(u) \setminus \{v\}$.
If the $r-2-\alpha \geq 2$ cops on $A$ all probe a distance of $1$, then the robber is on $v$.
Otherwise, the robber is on $N(b)\setminus \{v\}$ for exactly one $b \in N(v)$.
If only cop $C_i$ on $A$ probes a distance of $1$, then $b=C_i$.
If all cops on $A$ probe a distance of $3$, then $b$ is the unique vertex of $A$ without a cop.
Further, if we know the robber is in $N(b)\setminus \{v\}$ and one of the $\alpha+2$ cops on $N(v') \setminus \{u\}$, say cop $C_j$, probes a distance of $1$, then the robber is uniquely identified.
Each of the $\alpha+2$ cops on $N(v') \setminus \{u\}$ is disjoint to one $X_j$ for some $j$ in $1 \leq j  \leq r$.
A cop that is disjoint to $X_j$ is also adjacent to exactly one vertex in each $X_{j'}$ for $1 \leq j' \leq r$ and $j' \neq j$.
Also, each $X_j$ is disjoint to at most one cop on $N(v') \setminus \{u\}$.
This means that each set $X_j$ contains $(\alpha+2)-1 = \alpha+1$ vertices that are adjacent to cops in $N(v') \setminus \{u\}$, and so these $\alpha +1$ can be uniquely identified.
To have had avoided capture, the robber must have moved to one of the $(k-1)-(\alpha+1) = k -2 - \alpha$ vertices of $N(b)\setminus \{v\}$ that is not uniquely identifiable by the cops in $N(v') \setminus \{u\}$.
The cops, on their turn, identify that the robber is on a set of $k-2 - \alpha$ vertices $A' \subseteq N(u') \setminus \{v\}$ for some $v \in X$, $u' \in N(v)$. The remainder of the proof is identical to the end of the proof of Theorem~\ref{thm:pp}. \end{proof}

The previous theorem along with the lower bound of Theorem~\ref{thm:genLowerBIBD} yield a difference of one between this upper and lower bound:
\begin{corollary} \label{cor:near-symmetricBIBD}
Let $G$ be the incidence graph of a \rm{BIBD}$(k^2, k^2 + k,k+1,k,1)$ with $k \geq 3$. We then have that $\zeta(G)$ is either $k$ or $k+1$.
\end{corollary}

An \emph{affine plane of order} $q$ is a $\text{BIBD}$ with $k=q$ and $r =q+1$ such that the set of blocks partitions into \emph{parallel classes}. Corollary~\ref{cor:near-symmetricBIBD} shows that the localization number of an affine plane of order $q$ is either $q$ or $q+1$. We improve this to give an exact value.

\begin{theorem}\label{affinet}
Let $G$ be the incidence graph of an affine plane of order $k$ with $k \geq 3$. We then have that $\zeta(G) =  k$.
\end{theorem}
\begin{proof}
It remains to be shown that $\zeta(G) \leq k$; note that Theorem \ref{thm:genLowerBIBD} performs the lower bound.

We use the technique of decreasing the robber territory.
This proof is similar to the proof of Theorem \ref{thm:nearSymUpperBound}, and we will use some of the techniques found in that proof.
We provide a cop strategy with $k$ cops that secures the capture of the robber.
Let $B\subseteq \mathcal{B}$ be all the blocks from a parallel class, and so contains $k$ vertices.
We put $k-1$ cops on $B$, and the last cop on a vertex of $\mathcal{B} \setminus B$, and apply the scanning technique until the robber is forced to move to $X$ (if the robber remains on $\mathcal{B}$, then the robber's location will eventually be known exactly).
Once the robber moves, the cops are able to detect that the robber is on the $k$ points of $N(v)$, for some $v \in B$.
This is because either one of the cops on $B$ will probe a distance of $1$, in which case $v$ is the cop that probes $1$, or all the cops on $B$ probe a distance of $3$, in which case $v$ is the unique vertex of $B$ that does not contain a cop.
The cops can also capture the robber on one of the points on $N(v)$ (the one adjacent to the last cop, placed on $\mathcal{B} \setminus B$).
The cops now know that the robber is on a set of $k-1$ points, all of which are adjacent to a single vertex of $\mathcal{B}$.

\emph{Step 1: From $X$ to $\mathcal{B}$}. We can use the same cop strategy employed in Step 1 of Theorem~\ref{thm:pp}, however, the results are different as we have less cops.  We initialize $\alpha =0$, and enforce that $\alpha \geq -1$. Although $\alpha$ will typically increase, during the last stage of the proof $\alpha$ may be reduced. In Step 1, we suppose that the cops identify that the robber is on a set of $k-1-\alpha$ vertices $A \subseteq N(v)$ for some $v \in \mathcal{B}$, and $k-1 - \alpha \geq 3$.
Hence, the cops can identify that the robber is on a set of $k - (\alpha+1)$ vertices $A' \subseteq N(u') \setminus \{v\}$ for some $v \in \mathcal{B}$, $u' \in N(v)$.
Step 2 however differs in a significant manner.

\emph{Step 2: From $\mathcal{B}$ to $X$.}
Let $v \in \mathcal{B}$ and $u \in N(v)$.
Suppose that the cops identify that the robber is on a set of $k-1-\alpha$ vertices $A \subseteq N(u) \setminus \{v\}$, where $k-1 - \alpha \geq 3$.
The robber takes his turn.
The cops place $k-2 - \alpha$ cops on the $k-1 -\alpha$ vertices of $A$ without repetition.
Let $B$ be the parallel class that contains the block $v$.
The cops place a further $\alpha+2$ in vertices of $B \setminus\{v\}$.
Note that the blocks in $A$ and $B$ are not parallel, and so each vertex in $A$ and each vertex in $B$ have exactly one element in their common neighborhood.

If the $k-2-\alpha \geq 2$ cops on $A$ all probe a distance of $1$, then the robber is on $v$.
If cop $C_i$ is the only cop on $A$ to probe a distance of $1$, then the robber is on one of the $k-1$ vertices in $N(C_i) \setminus{u}$.
If all cops of $A$ probe a distance of $3$, then the robber is in $N(v') \setminus{u}$, where $v' \in A$ is the unique vertex in $A$ without a cop.

Each vertex in $B$ intersects  $N(v') \setminus{u}$ in exactly one vertex for each $v' \in A$.
Hence, if we know the robber is in $N(v')\setminus \{u\}$ for $v'\in A$ and a cop $C_j$ on $B$ probes a distance of $1$, then the robber is uniquely identified.
To avoid capture, the robber must have moved to one of the $k-1-(\alpha+2) = k - 3 -\alpha$ vertices in $N(v')\setminus \{u\}$ for some $v'\in A$ that are not adjacent to any vertex of $B$.

By applying an analogous argument to Theorem~\ref{thm:pp}, we perform Steps 1 and 2 until the robber is identified on a single vertex, and so the robber is captured.  \end{proof}

\section{Steiner Systems}\label{ss: steiner systems}

In this section, we consider the incidence graph of certain Steiner systems. For $v \ge 7,$ a \emph{Steiner triple system} is a BIBD with $\lambda =1$ and $k=3.$ For a Steiner triple system on $v$ points, denoted by $\rm{STS}(v)$, we have that $v \equiv 1$ or $3 \pmod 6$, $b = \frac{v(v-1)}{6},$ and $r = \frac{v-1}{2}$. Note that an \rm{STS}$(7)$ is isomorphic to a projective plane of order 2, and so has localization number 3 by Theorem~\ref{thm:pp}. An \rm{STS}$(9)$ is the affine plane of order 3, and so has localization number 3 by Theorem~\ref{affinet}. By Theorem~\ref{thm:genLowerBIBD}, we have the following.

\begin{corollary}\label{stsc}
For $v>9,$ the localization number of an \rm{STS}$(v)$ is strictly larger than $\lfloor \frac{v-2}{8} \rfloor$.
\end{corollary}

We can also provide an upper bound.

\begin{theorem}\label{stsv}
The localization number of an \rm{STS}$(v)$ is at most $\frac{v+1}{2}$.
\end{theorem}
\begin{proof}
We play with $\frac{v+1}{2}$ cops, and show that this is sufficient in order to capture the robber.
We divide the set of points $X$ into two disjoint sets $A$ and $B$ of size $\frac{v+1}{2}$ and $\frac{v-1}{2}$ respectively.
For the cops first move, we place a cop on each vertex of $A$.
If the robber is on a point in $A$ or if the robber is on a block that contains two points of $A$, the robber is captured.

In the case that the robber was on a block that contains three points of $B$, all cops probe a distance of three and so know that this is the case.
In response, for the next robber turn, the robber can only stay on his block, or move to a point of $B$.
In the cops next turn, the cops play on each vertex of $B$. If the robber is not on the same point as a cop, the robber remained on his block. In this case, exactly three cops will probe a distance of $1$ to the robber, which uniquely determines the block that the robber is on, and so the robber is captured.

In the case that the robber was on a block that contains two points of $B$ and one point of $A$, one single cop (say this cop is on point $p$) probed a distance of $1$, and so it is known that the robber is on a block that contains $p$.
The robber takes his turn.
During the cops' next turn, the cops play on each point in $B$ and on the point $p$ (this is $\frac{v+1}{2}$ points for $\frac{v+1}{2}$ cops).
If the robber moved to a point in $B$ or the point $p$, then he is captured (given that we are covering every neighbor of the candidate blocks except for $A$).
These were the only points adjacent to the robber, and so otherwise, the robber must be on a block. There are three cops that probe a distance of $1$ to the robber, which uniquely determines the block that the robber is on, and the robber is captured.

In the case that the robber was on a point of $B$, in their next turn, the cops play on the points of $B$.
The cops can capture the robber using the previous cases and a symmetric argument, except in the case that the robber moved from a point on $B$ to a block that contains two points of $A$ and one point of $B$.
In this case, one single cop on $B$ (say this cop is on point $p$) probed a distance of $1$, and so it is known that the robber is on a block that contains $p$.
The robber takes his turn.
During the cops' next turn, the cops choose a point $a \in A$ and the cops play on each point in $A\setminus \{a\}$ and on the point $p$ (this is $\frac{v+1}{2}$ points for $\frac{v+1}{2}$ cops).
If the robber moved to a point in $A\setminus{a}$ or the point $p$, then he is captured.
Also, if he moved to point $a$, then all the cops on $A\setminus{a}$ and $p$ probe a distance of $2$, which can only occur if the robber is on $a$, so then he is captured.
These were the only points adjacent to the robber, and so otherwise, the robber must be on a block. If there are three cops that probe a distance of $1$ to the robber, which uniquely determines the block that the robber is on, then the robber is captured.
Otherwise, there are two cops that probe a distance of $1$ to the robber, which uniquely determines the block that the robber is on, and so the robber is captured.
\end{proof}

We can also determine an asymptotic upper bound, that improves on Theorem~\ref{stsv} for large values of $v$. We use the notation $\ln{v}$ for the natural logarithm of $v.$

\begin{theorem} \label{thm:upperSTSAlon}
The localization number of an \rm{STS}$(v)$ is at most $v/3 + 2cv^{1/2}(\ln{v})^{3/2}+1,$ where $v/3 + 2cv^{1/2}(\ln{v})^{3/2}+1 \geq 9$ for some absolute constant $c$. In particular, the localization number of an \rm{STS}$(v)$ is at most $(1+o(1))v/3.$
\end{theorem}
\begin{proof}
Let $T$ be a maximum size subset of disjoint triples in the $\rm{STS}(v)$, and let $t = |T|$.
Let $Q$ be the (possibly empty) set of points in $X$ that do not occur in a triple of $T$.
It is known that $t \geq v/3 - cv^{1/2}(\ln{v})^{3/2}$, and so also $|Q| \leq 3cv^{1/2}(\ln{v})^{3/2}$; see \cite{AlonKimSpencer_NearlyPerfectMatchingsRegSimpHypergraphs}.
The main idea of this technique is to place a cop on each block of $T$ and place a cop on each point of $Q$, which uses $v/3 + 2cv^{1/2}(\ln{v})^{3/2}$ cops.
We will  then use another cop to perform the scanning technique. Some parts of the proof will also require the number of cops to be at least $9$.

In their first turn, the cops place a cop on each block of $T$ and place a cop on each point of $Q$.
If the robber is on a point of $Q$, then some cop probes a distance of $0$, so the robber is located. If the robber is on a block that contains three points $q_1,q_2,q_3$ of $Q$, then the three cops on $q_1,q_2,q_3$ are the only cops that probe a distance of $1$, and robber is located to be on the block $\{q_1,q_2,q_3\}$.
If the robber is on a block that contains two points $q_1,q_2$ of $Q$, then the two cops on $q_1,q_2$ are the only cops of $Q$ that probe a distance of $1$, then the robber is located to be on the unique block $\{q_1,q_2,p\}$, for some unique point $p$.

 If the robber is on a block that contains only one point $q$ of $Q$, then the cop on $q$ is the only cop of $Q$ that probes a distance of $1$.
 The block that contains the robber is then of the form $\{q,t_1,t_2\}$, where $t_1,t_2$ are points that each occur in triples of $T$.
 Suppose that the triples containing $t_1$ and $t_2$ are $\{t_1,r,s\}$ and $\{t_2,t,u\}$.
 We can assume without loss of generality that $\mathcal{B}$ contains the blocks $\{q,t_1,t_2\}$, $\{q,r,t\}$, or $\{q,s,u\}$.
 The cops know that the robber is on one of these blocks.
 In the robber's next move, he will only be able to stay on the same block, or move to one of the points in $\{q,t_1,t_2,r,s,t,u\}$.
 In the cops' next turn, they place seven cops on the points in $\{q,t_1,t_2,r,s,t,u\}$, which is able to capture the robber in either case.

If the robber is on a block that contains no points of $Q$, then the block that contains the robber is then of the form $\{t_1,t_2,t_3\}$, where $t_1,t_2,t_3$ are points that each occur in triples of $T$.
 Suppose that the triples containing $t_1$, $t_2$, and $t_3$ are $\{t_1,r,s\}$, $\{t_2,t,u\}$, and $\{t_3,v,w\}$.
 In the robber's next turn, he can either stay on the same triple, or move to a point in $\{t_1,t_2,t_3,r,s,t,u,v,w\}$.
 In the cops' next turn, they place nine cops on the points in $\{t_1,t_2,t_3,r,s,t,u,v,w\}$, which is able to capture the robber in either case.

 The only case left to consider is if the robber is on a point contained in a block $\{t_1,t_2,t_3\}$ of $T$.
 In this case, we assign an additional cop to probe the vertices $t_1$, $t_2$, and $t_3$ over the next three turns or until the robber moves to a block.
 If the robber does not move, then he will be captured. Otherwise, the robber moves to a block, and is captured using a previous case.
\end{proof}

A \emph{Steiner system} $S(t,k,v)$ with $2 \leq t <k<v$ is a set $X$ of size $v$ along with a set of blocks $\mathcal{B}$ with each block of size $k$ such that every $t$-subset of $V$ occurs in exactly one block.
A \emph{Steiner quadruple system} is an $S(3,4,v)$. The \emph{repetition number} of a point of an $S(t,v,k)$, written $D$, is the number of blocks that any one point occurs in. Note that $D=\binom{v-1}{t-1} / \binom{k-1}{t-1}.$

\begin{theorem}\label{stsq}
If $v\ge 6,$ then the localization number of an $S(3,4,v)$ is at most $v-3$.
\end{theorem}
\begin{proof}
Consider any three points from $X$, say $\left\{x, y, z\right\}$. The cops place themselves on every vertex in $X\setminus \left\{x, y, z\right\}$. We will show that the cops have a winning strategy from this placement.

If the robber is on a vertex in $\mathcal{B}$, then the three points $\left\{x, y, z\right\}$ are contained in a unique block of $\mathcal{B}$, say $(x,y,z,w)$. If the robber is located on this block, then the probe will return a distance of $1$ for the cop on $w\in X$ and a distance of $3$ for all other cops.
Next, if the robber is located on a block with a pair of points from $\left\{x, y, z\right\}$, suppose without loss of generality the block $\{x,y,a,b\}$ with $z \notin \{a,b\}$, then the cops on $a$ and $b$ both probe a distance of $1$ and all other cops probe a distance of $3$, which identifies the robber on residing on a block containing $\{a,b\}$ and two out of three points of $\left\{x, y, z\right\}$. There can only be one block that satisfies this constraint, and so the robber is located.

All other blocks have at least three cops at distance $1$ and thus, the location of the robber would be identified.
No matter which block the robber occupies, he will be identified by the cops with this placement.

We next consider the cases that the robber is on a vertex in $X$. If the robber is on a point occupied by a cop, then he is immediately identified. Otherwise, the cops send a second probe, this time replacing $\left\{x, y, z\right\}$ with three different vertices to omit, say  $\left\{x', y', z'\right\}$, and occupy all vertices of $X \setminus  \left\{x', y', z'\right\}$. This forces the robber off of his location, to a block in $\mathcal{B}$ which will be immediately identified by the above argument.\end{proof}

We finish the section with asymptotic bounds for Steiner systems.
\begin{theorem}\label{stsss}
The localization number of an $S(2,k,v)$ is at most $$v/k + O(v^{1-\frac{1}{k-1}})=(1+o(1))v/k.$$  The localization number an $S(t,k,v)$ is at most $$v/k + O(v \cdot D^{-\frac{1}{k-1}})=(1+o(1))v/k.$$
\end{theorem}

\begin{proof} The proof is analogous to Theorem~\ref{thm:upperSTSAlon}, using the result that any Steiner system $S(2, k, n)$ contains a matching that covers all vertices but at most $O(v^{1-\frac{1}{k-1}})$ and that any Steiner system $S(t, k, n)$ contains a matching that covers all vertices but at most $O(n\cdot D^{-\frac{1}{k-1}})$; see \cite{AlonKimSpencer_NearlyPerfectMatchingsRegSimpHypergraphs}.
\end{proof}

\section{Further directions}\label{ss: transversal designs}

We studied the localization number of incidence graphs of designs, and gave bounds on the localization number of balanced incomplete block designs. Exact values of the localization number are given for the incidence graphs of projective and affine planes. Bounds were given for Steiner systems.

There are a number of further directions when studying the localization number of graphs derived from designs. An interesting question is to find tight bounds on the localization number of families of Steiner systems. The localization number has yet to be considered for block intersection graphs, point graphs, or Latin square graphs. We will consider polarity graphs (which are graphs defined using polarities on projective planes) in the sequel.

We finish with a discussion of the localization number of transversal designs. A \emph{transversal design} $TD(k,n)$ of block size $k$ and group size $n$ is a pair $(X,\mathcal{B})$ such that the following hold.
\begin{enumerate}
\item $X$ is set of $kn$ points partitioned into groups $G_i$ for $1 \leq i \leq k$, with each group of size $n$;
\item $\mathcal{B}$ is a set of blocks of size $k$; and
\item Each block contains exactly one element in $G_i$, for each $1 \leq i \leq k$.
\end{enumerate}

Famously, the existence of a $TD(n+1,n)$ is equivalent to the existence of a projective plane of order $n$: given a projective plane of order $q$, if we pick a point and remove all blocks that contain that point, then we have a $TD(q+1,q)$. In a similar fashion, if we remove any parallel class in an affine plane of order $q$, then we are left with a $TD(q,q)$. We note that although we know the exact localization numbers of affine planes and projective planes, this does not imply anything directly about the localization numbers of a $TD(k,n)$, as the removal or addition of blocks and points has an unknown effect. We do have the following upper bound.

\begin{theorem}\label{td}
If $G$ is the incidence graph of a $TD(k,n)$ with $k\geq 4$, then $\zeta(G) \leq n+k-4$.
\end{theorem}
\begin{proof}
We play the game with $n+k-4$ cops, and show that the robber can be captured.
Note that there are at least $n$ cops.
For their first move, the cops place $n-1$ cops on the $n$ vertices of $X$ that correspond to a group $G_1$ of the transversal design.
One further cop is placed on another vertex of $X$.
Until the robber is detected on $\mathcal{B}$ (by any cop probing an odd distance to the robber), this last cop moves along the vertices of $X$, forcing the robber to move to $\mathcal{B}$ at some point, by the scanning technique.

When the robber moves to $\mathcal{B}$, there is exactly one vertex $u\in G_1$ that has distance $1$ to the robber. If a cop is on $u$, then this cop probes a distance of $1$ and all others on $G_1$ probe a distance of $3$.
Otherwise, all cops on $G_1$ probe a distance of $3$, and $u$ is the unique vertex in $G_1$ without a cop.
The robber is discovered to be on a vertex of $N(u)$.
The robber takes his turn.

The cop player places $n$ cops on the $n$ vertices in $N(u)$.
If all of these cops probe a distance of $1$ to the robber, then the robber is captured on $u$.
Otherwise, exactly one of these cops will probe a distance of $1$ to the robber, say at a vertex $v \in N(u)$.
Thus, the robber is known to reside on one of the $k-1$ vertices in $N(v)\setminus \{u\}$.
The robber takes his turn. 

The cops place $k-2$ cops on the $k-1$ vertices in $N(v)\setminus \{u\}$ and $n-2$ cops on the $n-1$ vertices of $G_1 \setminus \{u\}$.
If the robber did not move from $N(v)$, then he is captured as either some cop on $N(v)$ probes a distance of $0$, or all cops on $N(v)$ probe a distance of $2$ implying the robber is on the unique vertex of $N(v)$ without a cop.
As before, the cops of $N(v)$ can identify the unique $u' \in N(v)\setminus \{u\}$ such that the robber is on $N(u')$.
The cops on $G_1$ can also identify the unique $u'' \in G_1\setminus \{u\}$ such that the robber is on $N(u'')$.
Note that $N(v)\setminus \{u\}$ and $G_1\setminus \{u\}$ are disjoint subsets of $X$.
Hence, any vertex in  $N(v)\setminus \{u\}$ has distance $2$ from any vertex in $G_1\setminus \{u\}$ and further there is exactly one vertex of $\mathcal{B}$ in their common neighborhood, as each pair of points in different groups of the transversal design are in exactly one block.
Hence, the robber is located as being on the unique vertex in $N(u') \cap N(u'')$.
\end{proof}

We do not know if the bound in Theorem~\ref{td} is tight.


\begin{thebibliography}{99}

\bibitem{AlonKimSpencer_NearlyPerfectMatchingsRegSimpHypergraphs} N.~Alon, J.~Kim, J.~Spencer, Nearly perfect matchings in regular simple hypergraphs, \emph{Israel Journal of Mathematics} \textbf{100} (1997) 171--187.

\bibitem{bahl} P.\ Bahl, V.N.\ Padmanabhan, RADAR: an in-building RF-based user location and tracking system, In: \emph{Proceedings of INFOCOM 2000, Nineteenth Annual Joint Conference of the IEEE Computer and Communications Societies}, 2000.

\bibitem{bailey} R.F.\ Bailey, On the metric dimension of incidence graphs, \emph{Discrete Mathematics} \textbf{341} (2018) 1613--1619.

\bibitem{BB} A.~Bonato, A.\ Burgess, Cops and Robbers on graphs based on designs  {\em Journal of Combinatorial Designs} \textbf{21} (2012) 404--418.

\bibitem{BK} A.\ Bonato, W.\ Kinnersley, Bounds on the localization number, \emph{J.\ Graph Theory} (2020) 1--18.

\bibitem{bm} A.~Bonato, B.\ Mohar, Topological directions in Cops and Robbers, \emph{Journal of Combinatorics} \textbf{11} (2020) 47--64.

\bibitem{BN} A.~Bonato, R.J.\ Nowakowski, {\em The Game of Cops and Robbers on Graphs}, American Mathematical Society, Providence, RI, 2011.

\bibitem{bp} A.~Bonato, P.\ Pra\l{}at, \emph{Graph Searching Games and Probabilistic Methods}, CRC Press, 2017.

\bibitem{by} A.\ Bonato, B.\ Yang, Graph searching and related problems, invited book chapter in: \emph{Handbook of Combinatorial Optimization}, editors P. Pardalos, D.Z. Du, R. Graham, 2011.

\bibitem{nisse1} B.\ Bosek, P.\ Gordinowicz, J.\ Grytczuk, N.\ Nisse, J.\ Sok\'o\l, M.\ \'Sleszy\'nska-Nowak, Localization game on geometric and planar graphs, \emph{Discrete Applied Mathematics} \textbf{251} (2018) 30--39.

\bibitem{nisse2} B.\ Bosek, P.\ Gordinowicz, J.\ Grytczuk, N.\ Nisse, J.\ Sok\'o\l, M.\ \'Sleszy\'nska-Nowak, Centroidal localization game, \emph{Electronic J. Combin.} \textbf{25} no. 4 (2018), article P4.62.

\bibitem{BDELM} A.\ Brandt, J.\ Diemunsch, C.\ Erbes, J.\ LeGrand, C.\ Moffatt, A robber locating strategy for trees, \emph{Discrete Applied Mathematics} \textbf{232} (2017) 99--106.

\bibitem{car} J.\ Carraher, I.\ Choi, M.\ Delcourt, L.H.\ Erickson, D.B.\ West, Locating a robber on a graph via distance queries, \emph{Theor. Computer Science} \textbf{463} (2012) 54--61.

\bibitem{HCD} C.~Colbourn and J.~Dinitz. Handbook of Combinatorial Designs, Second Edition \emph{Discrete Mathematics and Its Applications}, Chapman \& Hall/CRC (2006).

\bibitem{DEFMP} A.\ Dudek, S.\ English, A.\ Frieze, C.\ MacRury, P.\ Pra\l{}at, Localization game for random graphs, Preprint 2020.

\bibitem{DFP} A.\ Dudek, A.\ Frieze, W.\ Pegden, A note on the localization number of random graphs: diameter two case,
 {\em Discrete Applied Mathematics} \textbf{254} (2019) 107--112.

\bibitem{fomin} F.V.\ Fomin, D.M.\ Thilikos, An annotated bibliography on guaranteed graph searching, \emph{Theor. Computer Science} \textbf{399} (2008) 236--245.

\bibitem{frankl} P.\ Frankl, Cops and robbers in graphs with large girth and Cayley graphs, \emph{Discrete Applied Mathematics} \textbf{17} (1987) 301--305.

\bibitem{hm} F.\ Harary, R.A.\ Melter, On the metric dimension of a graph, \emph{Ars Combin.} \textbf{2} (1976) 191--195.

\bibitem{has} J.\ Haslegrave, R.A.B.\ Johnson, S.\ Koch, Locating a robber with multiple probes, \emph{Discrete Mathematics} \textbf{341} (2018) 184--193.

\bibitem{ht} T.\ H\'eger,  M.\ Tak\'ats, Resolving sets and semi-resolving sets in finite projective planes, \emph{Electronic J. Combin.} \textbf{19}:4 (2012), \#P30.

\bibitem{p} P.\ Pra\l{}at, When does a random graph have a constant cop number? \emph{Aust. J. Combin.} \textbf{46} (2010) 285--296.

\bibitem{seager1} S.\ Seager, Locating a robber on a graph, \emph{Discrete Mathematics} {\bf 312} (2012) 3265--3269.

\bibitem{seager2} S.\ Seager, Locating a backtracking robber on a tree, \emph{Theor. Computer Science} {\bf 539} (2014) 28--37.

\bibitem{slater} P.J.\ Slater, Leaves of trees, In: \emph{Proc. Sixth Southeastern Conf. Combin., Graph Theory, Computing, Congressus Numer.} \textbf{14} (1975) 549--559.

\bibitem{Stinson2004} D.R.~Stinson, {\em Combinatorial Designs: Constructions and Analysis,} Springer-Verlag New York, Inc., 2004.

\bibitem{West} D.B.\ West, \emph{Introduction to Graph Theory, 2nd edition}, Prentice Hall, 2001.


\end{thebibliography}
\end{document}